\documentclass[12pt,reqno]{article}

\usepackage[usenames]{color}
\usepackage{amssymb}
\usepackage{amsmath}
\usepackage{amsthm}
\usepackage{amsfonts}
\usepackage{amscd}
\usepackage{graphicx}

\usepackage[colorlinks=true,
linkcolor=webgreen,
filecolor=webbrown,
citecolor=webgreen]{hyperref}

\definecolor{webgreen}{rgb}{0,.5,0}
\definecolor{webbrown}{rgb}{.6,0,0}

\usepackage{color}
\usepackage{fullpage}
\usepackage{float}

\usepackage{graphics}
\usepackage{latexsym}

\begin{document}

\theoremstyle{plain}
\newtheorem{theorem}{Theorem}
\newtheorem{proposition}{Proposition}
\newtheorem{corollary}[theorem]{Corollary}
\newtheorem{lemma}{Lemma}
\newtheorem*{example}{Example}
\newtheorem{remark}{Remark}

\begin{center}
\vskip 1cm{\LARGE\bf 
Convolutions of second order sequences: \\ A direct approach \\
}
\vskip 1cm
{\large
Kunle Adegoke \\
Department of Physics and Engineering Physics, \\ Obafemi Awolowo University, Ile-Ife \\ Nigeria \\
\href{mailto:adegoke00@gmail.com}{\tt adegoke00@gmail.com}

\vskip 0.2 in

Segun Olofin Akerele \\
Department of Mathematics \\ University of Ibadan, Ibadan \\ Nigeria \\
\href{mailto:akereleolofin@gmail.com}{\tt akereleolofin@gmail.com}

\vskip 0.2 in

Robert Frontczak \\
Independent Researcher, Reutlingen \\  Germany \\
\href{mailto:robert.frontczak@web.de}{\tt robert.frontczak@web.de}}
\end{center}

\vskip .2 in

\begin{abstract}
Using a direct algebraic approach we derive convolution identities for second order sequences,
hereby distinguishing between sequences obeying the same or different recurrence relations.
We also state a general convolution for Horadam sequences. Convolutions for Chebyshev polynomials will also be stated.
\end{abstract}

\medskip



\medskip

\section{Motivation and introduction}

Let $(F_n)_{n\in\mathbb{Z}}$, $(L_n)_{n\in\mathbb{Z}}$, $(J_n)_{n\in\mathbb{Z}}$, $(j_n)_{n\in\mathbb{Z}}$, $(P_n)_{n\in\mathbb{Z}}$, 
$(Q_n)_{n\in\mathbb{Z}}$, $(B_n)_{n\in\mathbb{Z}}$ and $(C_n)_{n\in\mathbb{Z}}$ denote the $n$th Fibonacci, Lucas, Jacobsthal, 
Jacobsthal-Lucas, Pell, Pell-Lucas, balancing and Lucas-balancing numbers, respectively. \\

Convolutions are ubiquitous in number theory. They exist for many important number and polynomial sequences 
and are usually derived via manipulations of the respective generating functions. 
Classical examples involving Fibonacci and Lucas numbers are \cite{Koshy,kim16}
\begin{equation}
\sum_{k=0}^{n} F_{k}F_{n-k} = \frac{1}{5}((n+1)L_{n} - 2F_{n+1}),
\end{equation}
\begin{equation}
\sum_{k=0}^{n} L_{k}L_{n-k} = (n+1)L_{n} + 2F_{n+1}, 
\end{equation}
\begin{equation}
\sum_{k=0}^{n} L_{k} F_{n-k} = (n+1)F_{n}.
\end{equation}
Other examples for mixed convolutions are
\begin{equation}
\sum_{k=0}^n J_k F_{n-k} = J_{n+1} - F_{n+1} \quad (\mbox{see}\,\,\cite{Koshy18}),
\end{equation}
\begin{equation}
\sum_{k=0}^n P_k F_{n-k} = P_{n} - F_{n} \quad (\mbox{see}\,\,\cite{Greubel,Seiffert}),
\end{equation}
\begin{equation}\label{LJ_conv}
\sum_{k = 0}^n L_k J_{n - k} = j_{n + 1} - L_{n + 1} \quad (\mbox{see}\,\,\cite{griffiths15,bramham16})
\end{equation}
\begin{equation}
\sum_{k=0}^{n} F_{2k} B_{2(n-k)} = \frac{1}{31} \Big ( B_{2n} - 6F_{2n} \Big ) \quad (\mbox{see}\,\,\cite{frontczak19}).
\end{equation}
For a more general treatment of convolutions relevant for this study see the papers 
\cite{Dresden21,Dresden22,Dresden24,frontczak2022,frontczak2024,Gessel24,Merca,szakacs16,szakacs17}.  \\

This paper offers a different and more direct access to convolutions for second order sequences
through the manipulation of various algebraic identities. The approach is very intuitive 
as it does not require any knowledge about generating functions. Among the many results that will be derived, 
we will prove the following generalization of \eqref{LJ_conv} (Theorem \ref{thm.lwh48n4} below):
For all integers $r$ and $n$, we have the convolution
\begin{equation*}
\begin{split}
\gamma (r)\sum_{k = 0}^n L_{rk} J_{r(n - k)}  &= (- 1)^r (L_r - j_r )L_{rn} J_r \\
&\qquad - \left( \left( {( - 1)^r + j_{2r} } \right)J_r - J_{3r} \right)L_{r(n + 1)} \\
&\qquad + (- 1)^r 2^r J_{rn} \left( (- 1)^r 2^{r + 1} - L_r j_r + L_{2r} \right) \\
&\qquad - (- 1)^r J_{r(n + 1)} \left( \left( L_r - 2 \right ) j_r + \left( 1 - ( - 1)^r \right )L_r \right),
\end{split}
\end{equation*}
where
\[
\gamma (r) = j_{2r} \left( 1 + (- 1)^r \right) + (- 1)^r 2^r L_r (L_r - j_r ) - (- 1)^r L_r j_r. 
\]
 
As will be seen below, our approach is also applicable for proving convolutions for Chebyshev polynomials.
Such convolutions have also been studied recently by Fan and Chu \cite{Fan}. \\

The Horadam sequence $(w_n)_{n\in\mathbb{Z}}=(w_n(a,b;p,q))_{n\in\mathbb{Z}}$ is defined, for all integers, 
by the recurrence relation \cite{Horadam} 
\[
w_0=a, \ w_1=b, \ \ w_n=pw_{n-1}-qw_{n-2}, \quad n\geq 2,
\]
with 
\[
w_{-n}=\frac{1}{q}(pw_{-n+1}-w_{-n+2}),
\]
where $a,b,p$ and $q$ are arbitrary complex numbers with non-zero $p$ and $q$. The sequence $w_n$ generalizes many 
important number and polynomial sequences. Sequences 
\[
U_n = U_n(p,q) = w_n(0,1;p,q) \quad \text{and} \quad V_n = V_n(p,q) = w_n(2,p;p,q)
\]
are called the Lucas sequences of the first kind and of the second kind, respectively. The Binet formulas for the sequences 
$U_n, V_n$ and $w_n$ in the non-degenerated case, $\Delta=p^2-4q>0$, are 
\[
U_n = \frac{\alpha^n-\beta^n}{\alpha-\beta}, \quad V_n=\alpha^n + \beta^n, \quad 
w_n =\frac{b-a\beta}{\alpha-\beta}\alpha^n + \frac{a\alpha-b}{\alpha-\beta}\beta^n,
\]
where 
\[
\alpha = \alpha(p,q) = \frac{p+\sqrt{p^2-4q}}{2} \qquad \text{and} \qquad \beta = \beta(p,q) = \frac{p-\sqrt{p^2-4q}}{2}
\]
are the distinct zeros of the characteristic polynomial $x^2-px+q$.\\
For an integer $n\geq 0$, the Chebyshev polynomials $(t_n(x))_{n\geq 0}$ of the first kind are defined by 
\[
t_0(x)=1, \, t_1(x)=x, \quad t_{n+1}(x) = 2x t_n(x) - t_{n-1}(x),
\]
while the Chebyshev polynomials $(u_n(x))_{n\geq 0}$ of the second kind are defined by
\[
u_0(x)=1, \ \ u_1(x)=2x, \ \ u_{n+1}(x) = 2xu_n(x) - u_{n-1}(x).
\]
Chebyshev polynomials are special Horadam sequence members and their Binet-like formulas are given by
\begin{align*}
	t_n(x) &= \frac{1}{2} \left((x+\sqrt{x^2-1})^n + (x-\sqrt{x^2-1})^n \right), \\
	u_n(x) &= \frac{1}{2\sqrt{x^2-1}} \left((x+\sqrt{x^2-1})^{n+1} - (x-\sqrt{x^2-1})^{n+1}\right).
\end{align*}
from which we also get 
\[
t_n(x) = \frac{1}{2}(u_n(x) - u_{n-2}(x)).
\]
\vspace{1.5mm}
Our starting point is the identity stated in the next Theorem.
\begin{theorem}\label{thm.n3pubdo}
Let $A_1$, $A_2$, $B_1$, $B_2$, $x$, $y$, $z$ and $w$ be non-zero complex numbers. Then,
\begin{equation}
\begin{split}
&\sum_{k = 0}^n \left( {A_1 x^k  + B_1 y^k } \right)\left( {A_2 z^{n - k} + B_2 w^{n - k} } \right) \\ 
&\qquad = A_1 A_2 \frac{{x^{n + 1} - z^{n + 1} }}{{x - z}} + A_1 B_2 \frac{{x^{n + 1} - w^{n + 1} }}{{x - w}} 
+ A_2 B_1 \frac{{y^{n + 1} - z^{n + 1} }}{{y - z}} + B_1 B_2 \frac{{y^{n + 1} - w^{n + 1} }}{{y - w}}.
\end{split}
\end{equation}
\end{theorem}
\begin{proof}
The proof is a straightforward calculation using the partial sum of the geometric progression.
\end{proof}
Theorem \ref{thm.n3pubdo} contains a range of interesting algebraic identities as special cases.
These identities will be utilized to provide direct derivations for convolutions.

\section{Convolutions of two second order sequences}

In this section derive convolutions of two sequences obeying the same or different recurrence relations.

\begin{theorem}\label{thm.qn6huxt}
Let $A_1$, $A_2$, $B_1$, $B_2$, $z$ and $w$ be non-zero complex numbers. Then,
\begin{equation}
\begin{split}
&\sum_{k = 0}^n \left( A_1 z^k + B_1 w^k \right)\left( A_2 z^{n - k} + B_2 w^{n - k} \right) \\ 
&\qquad = (n + 1)\left( A_1 A_2 z^n + B_1 B_2 w^n \right) + \left( A_1 B_2 + A_2 B_1 \right) \frac{{z^{n + 1} - w^{n + 1} }}{{z - w}}.
\end{split}
\end{equation}
\end{theorem}
\begin{proof}
Set $x=z$, $y=w$ in the identity stated in Theorem~\ref{thm.n3pubdo} and use the L'Hospital's rule.
\end{proof}

\begin{example}
Convolution of the Pell and Pell-Lucas numbers, $P_n$ and $Q_n$;
\begin{equation}
\sum_{k=0}^n P_{rk} Q_{r(n-k)} = (n+1) P_{rn}.
\end{equation}
\begin{proof}
In Theorem~\ref{thm.qn6huxt}, make the following choices:
\[
A_1 = \frac{1}{a-b}, \ \ z=a^r, \ \ B_1=-\frac{1}{a-b}, \ \ w=b^r, \ \ A_2=B_2=1,
\]
where $a=1+\sqrt{2}$, $b=1-\sqrt{2}$. Recall that the Pell and Pell-Lucas numbers are given by
\[
P_n = \frac{a^n - b^n}{a-b}, \qquad Q_n = a^n + b^n.
\]
\end{proof}
\end{example}
\begin{example}
Convolution of the Jacobsthal and Jacobsthal-Lucas numbers, $J_m$ and $j_m$;
\begin{equation}
\sum_{k=0}^n J_{rk} j_{r(n-k)} = (n+1) J_{rn}
\end{equation}
\begin{proof}
In Theorem~\ref{thm.qn6huxt}, make the following choices:
\[
A_1=\frac{1}{3}, \ \ z=2^r, \ \ B_1=-\frac{1}{3}, \ \ w=(-1)^r, \ \ A_2=B_2=1,
\]
where $J_m=\frac{2^m-(-1)^m}{3}$ and $j_m=2^m+(-1)^m$.	
\end{proof}
\end{example}
\begin{example}
Convolution of the balancing and Lucas-balancing numbers, $B_n$ and $C_n$;
\begin{equation}
\sum_{k=0}^n B_{rk} C_{r(n-k)} = \frac{(n+1)}{2} B_{rn}.
\end{equation}
\begin{proof}
In Theorem~\ref{thm.qn6huxt}, make the following choices:
\[
A_1 = \frac{1}{\lambda_1-\lambda_2}, \ \ z=\lambda_1^r, \ \ B_1=-\frac{1}{\lambda_1-\lambda_2}, \ \ w=\lambda_2^r, \ \ A_2=B_2=1/2,
\]
where $\lambda_1=3+\sqrt{8}$, $\lambda_2=3-\sqrt{8}$. Recall that the balancing and Lucas-balancing numbers are given by
\[
B_n = \frac{\lambda_1^n - \lambda_2^n}{\lambda_1-\lambda_2}, \qquad C_n = \frac{\lambda_1^n + \lambda_2^n}{2}.
\]
\end{proof}
\end{example}

\begin{example}
Convolution of the Chebyshev polynomials of the first and second kind, $t_n(x)$ and $u_n(x)$:
\begin{equation}
2u_{r - 1}(x) \sum_{k = 0}^n t_{rk}(x) u_{r(n - k)}(x) = (n + 1)u_{r - 1}(x) u_{rn}(x) + u_{rn + r - 1}(x).
\end{equation}
\end{example}
\begin{proof}
In Theorem~\ref{thm.qn6huxt}, make the following choices:
\[
A_1 = \frac{x}{{\lambda + \gamma }},\quad B_1 = \frac{x}{{\lambda + \gamma }},\quad A_2 = \frac{\lambda }{{\lambda - \gamma }},\quad B_2 = \frac{{- \gamma }}{{\lambda - \gamma }},\quad z = \lambda^r,\quad w = \gamma^r,
\]
where $\lambda$ and $\gamma$ are the roots of $1-2sx+s^2=0$, that is, $\lambda=x + \sqrt{x^2 -1}$ and $\gamma=x - \sqrt{x^2 -1}$. 
Recall that the Chebyshev polynomials are given by the Binet-like formulas
\[
t_n (x) = x\frac{{\lambda^n + \gamma^n }}{{\lambda + \gamma }},\quad u_n (x) = \frac{{\lambda^n - \gamma^n }}{{\lambda - \gamma}}.
\]
\end{proof}

\begin{corollary}\label{corr.y}
Let $A$, $B$, $z$ and $w$ be non-zero complex numbers. Then,
\begin{equation}
\sum_{k = 0}^n \left( Az^k + Bw^k \right)\left( Az^{n - k} - Bw^{n - k} \right) = (n + 1)\left( A^2 z^n - B^2 w^n \right).
\end{equation}
\end{corollary}
\begin{proof}
In Theorem~\ref{thm.qn6huxt}, set $A_1=A_2=A$, $B_1=B=-B_2$.
\end{proof}

\begin{example}
Convolution of Fibonacci numbers and Lucas numbers
\begin{equation}
\sum_{k = 0}^n L_{rk} F_{r(n - k)} = (n + 1)F_{rn}.
\end{equation}
More generally, convolution of Lucas sequence of the first kind and of the second kind:
\begin{equation}\label{UV_conv}
\sum_{k = 0}^n U_{rk} V_{r(n - k)} = (n + 1) U_{rn}.
\end{equation}
\end{example}

\begin{remark}
Similar convolutions have been studied recently by Dresden and Wang \cite{Dresden24}.
\end{remark}

\begin{example}
\begin{equation}
\sum_{k=0}^n t_{rk}(x)u_{r(n-k)}(x) = \frac{x(n+1)}{\lambda + \gamma}u_{rn}(x)
\end{equation}
\begin{proof}
Set $A=B=x/(\lambda + \gamma)$ and $z=\lambda^r$ and $w=\gamma^r$ in Corollary~\ref{corr.y}. 
\end{proof}
\end{example}

\begin{corollary}\label{cor.ytalcuq}
Let $A$, $B$, $z$ and $w$ be non-zero complex numbers. Then,
\begin{equation}
\sum_{k = 0}^n \left( A z^k + B w^k \right) \left( A z^{n - k} + B w^{n - k} \right) 
= (n + 1)\left( A^2 z^n + B^2 w^n \right) + 2AB \frac{{z^{n + 1} - w^{n + 1} }}{{z - w}}.
\end{equation}
\end{corollary}
\begin{proof}
In Theorem~\ref{thm.qn6huxt}, set $A_1=A_2=A$, $B_1=B=B_2$.
\end{proof}

\begin{example}
Self convolution of Fibonacci numbers and Lucas numbers:
\begin{equation}
5F_r \sum_{k = 0}^n F_{rk} F_{r(n - k)} = (n + 1)F_r L_{rn} - 2F_{r(n + 1)},
\end{equation}
\begin{equation}
F_r \sum_{k = 0}^n L_{rk} L_{r(n - k)} = (n + 1)F_r L_{rn} + 2F_{r(n + 1)}. 
\end{equation}
More generally, self convolution of Lucas sequences:
\begin{equation}\label{UU_conf}
U_r \Delta \sum_{k = 0}^n U_{rk} U_{r(n - k)} = (n + 1)U_r V_{rn} - 2U_{r(n + 1)},
\end{equation}
\begin{equation}\label{VV_conf}
U_r \sum_{k = 0}^n V_{rk} V_{r(n - k)} = (n + 1)U_r V_{rn} + 2U_{r(n + 1)}. 
\end{equation}
\end{example}

\begin{example}
Another example is the self convolution of the Chebyshev polynomials, namely,
\begin{equation}\label{eq.qa5lyvu}
2u_{r - 1}(x) \sum_{k = 0}^n t_{rk}(x) t_{r(n - k)}(x) = (n + 1)u_{r - 1}(x) t_{rn}(x) + u_{rn + r - 1}(x)
\end{equation}
and
\begin{equation}\label{eq.h4s17al}
2(x^2 - 1)u_{r - 1}(x) \sum_{k = 0}^n u_{rk}(x) u_{r(n - k)}(x) = (n + 1)u_{r - 1}(x) t_{rn + 2}(x) - u_{rn + r - 1}(x).
\end{equation}
Identity~\eqref{eq.qa5lyvu} was obtained by setting $A=B=x/(\lambda + \gamma)$ and $z=\lambda^r$ and $w=\gamma^r$ 
in Corollary~\ref{cor.ytalcuq}; while~\eqref{eq.h4s17al} was produced with the choice $A=\lambda/(\lambda - \gamma)$, 
$B=-\gamma/(\lambda - \gamma)$ and $z=\lambda^r$ and $w=\gamma^r$.
\end{example}

\begin{corollary}
Let $A$, $B$, $z$ and $w$ be non-zero complex numbers. Then,
\begin{equation}
\sum_{k = 0}^n \left( Az^{2k} + Bw^{2k} \right)\left( Az^{2n - 2k} - Bw^{2n - 2k} \right) 
= (n + 1)\left( Az^n + Bw^n \right)\left( Az^n - Bw^n \right).
\end{equation}
\end{corollary}
\begin{proof}
In Theorem~\ref{thm.qn6huxt}, set $A_1=A_2=A$, $B_1=B=-B_2$ and write $z^2$ for $z$ and $w^2$ for $w$.
\end{proof}

\begin{example}
For the Lucas sequences of the first and of the second kinds we have the convolution:
\begin{equation}
\sum_{k=0}^n U_{2r(n-k)} V_{2rk} = (n+1)U_{rn} V_{rn}.
\end{equation}
\end{example}

\begin{example}
For the Chebyshev polynomials, we have
\begin{equation}
\sum_{k=0}^n t_{2rk}(x) u_{2r(n-k)}(x) = (n+1) t_{rn}(x) u_{rn}(x).
\end{equation}
\end{example}

\begin{corollary}
Let $A$, $B$, $z$ and $w$ be non-zero complex numbers. Then,
\begin{align}
&\sum_{k = 0}^n \left( {Az^{2k} + Bw^{2k} } \right)\left( {Az^{2n - 2k} + Bw^{2n - 2k} } \right) \nonumber \\ 
&\qquad = (n + 1)\left( {Az^n + Bw^n } \right)^2 - 2(n + 1)AB(zw)^n + 2AB\frac{{z^{2n + 2} - w^{2n + 2} }}{{z^2 - w^2 }}.
\end{align}
\end{corollary}
\begin{proof}
In Theorem~\ref{thm.qn6huxt}, set $A_1=A_2=A$, $B_1=B_2=B$ and write $z^2$ for $z$ and $w^2$ for $w$.
\end{proof}

\begin{example}
For the Lucas sequences of the second kind we have the convolution:
\begin{equation}
\sum_{k=0}^n V_{2rk} V_{2r(n-k)} = (n+1) \left ( V_{rn}^2 - 2 q^{rn} \right ) + 2 \frac{U_{2r(n+1)}}{U_{2r}}.
\end{equation}
\end{example}

\begin{example}
For the Chebyshev polynomials, we have
\begin{equation}
\sum_{k=0}^n t_{2rk}(x) t_{2r(n-k)}(x) = (n+1) \left (t_{rn}^2(x) - \frac{1}{2}\right ) + \frac{1}{2}\frac{u_{2r(n+1)-1}(x)}{u_{2r-1}(x)}.
\end{equation}
\end{example}

\section{Convolution of two Horadam sequences}

\begin{theorem}\label{thm.rtf6ne1}
For $j\in\mathbb Z$, let $X_j(p_X,q_X)$ and $Y_j(p_Y,q_Y)$ be two Horadam sequences with respective associated 
Lucas sequences of the second kind ${V_X}_j(p_X,q_X)$ and ${V_Y}_j(p_Y,q_Y)$. Then,
\begin{equation}
\begin{split}
\gamma(r) \sum_{k = 0}^n X_{rk} Y_{r(n - k)} &= q_X^r X_{rn} \left( (q_X^r - {V_X}_r {V_Y}_r + {V_Y}_{2r} )Y_0 + {V_X}_r Y_r - Y_{2r} \right)\\
&\quad - X_{r(n + 1)} \left( {(q_X^r - {V_X}_r  {V_Y}_r + {V_Y}_{2r}  )Y_r + {V_X}_r Y_{2r} - Y_{3r} } \right)\\
&\quad\, + q_Y^r Y_{rn} \left( {(q_Y^r - {V_Y}_r {V_X}_r + {V_X}_{2r} )X_0 + {V_Y}_r X_r - X_{2r} } \right)\\
&\quad\; - Y_{r(n + 1)} \left( {(q_Y^r - {V_Y}_r {V_X}_r + {V_X}_{2r} )X_r + {V_Y}_r X_{2r} - X_{3r} } \right);
\end{split}
\end{equation}
where
\begin{equation}
\gamma (r) = q_X^{2r} + q_Y^{2r} - (q_X^r + q_Y^r ){V_X}_r {V_Y}_r + q_X^r {V_Y}_{2r} + q_Y^r {V_X}_r^2. 
\end{equation}
\end{theorem}
\begin{proof}
Let $x=\alpha_1^r$, $y=\beta_1^r$, $z=\alpha_2^r$, $w=\beta_2^r$. Consider Horadam sequences $(X_k)$ and $(Y_k)$, $k\in\mathbb Z$, where
\[
X_k = A_1 \alpha _1^k + B_1 \beta _1^k ,\quad Y_k = A_2 \alpha _2^k + B_2 \beta _2^k, 
\]
with
\[
{V_X} _k = \alpha _1^k + \beta _1^k ,\quad {V_Y}_k = \alpha _2^k + \beta _2^k.
\]
Use these in Theorem~\ref{thm.n3pubdo}.
\end{proof}

\begin{theorem}\label{thm.lwh48n4}
For all integers $r$ and $n$,
\begin{equation}
\begin{split}
\gamma (r)\sum_{k = 0}^n {L_{rk} J_{r(n - k)} }  &= ( - 1)^r (L_r - j_r )L_{rn} J_r\\
&\qquad - \left( {\left( {( - 1)^r + j_{2r} } \right)J_r - J_{3r} } \right)L_{r(n + 1)}\\
 &\qquad+\, ( - 1)^r 2^r J_{rn} \left( {( - 1)^r 2^{r + 1} - L_r j_r + L_{2r} } \right)\\
&\qquad\; - ( - 1)^r J_{r(n + 1)} \left( {\left( {L_r - 2} \right)j_r + \left( {1 - ( - 1)^r } \right)L_r } \right),
\end{split}
\end{equation}
where
\[
\gamma (r) = j_{2r} \left( {1 + ( - 1)^r } \right) + ( - 1)^r 2^r L_r (L_r  - j_r ) - ( - 1)^r L_r j_r. 
\]
\end{theorem}
\begin{proof}
In Theorem~\ref{thm.rtf6ne1} set $X_m=L_m$, $Y_m=J_m$, ${V_X}_m=L_m$, ${V_Y}_m=j_m$, $q_X=-1$ and $q_Y=-2$.
\end{proof}

\begin{example}
For integer $n$,
\begin{equation}
\sum_{k = 0}^n L_k J_{n - k} = j_{n + 1} - L_{n + 1}, 
\end{equation}
\begin{equation}
\sum_{k = 0}^n L_{2k} J_{2n - 2k} = J_{2n + 2} - F_{2n + 2}, 
\end{equation}
and
\begin{equation}
\sum_{k = 0}^n L_{3k} J_{3n - 3k} = \frac{1}{{62}}(11 J_{3n + 3} + 104 J_{3n}) - \frac{3}{{124}}(7L_{3n + 3} - 3L_{3n} ).
\end{equation}
\end{example}

\begin{theorem}
For all integers $r$ and $n$,
\begin{equation}
\begin{split}
\gamma(r)\sum_{k = 0}^n F_{rk}P_{r(n-k)} &= (-1)^rF_{rn}(L_rP_r-P_{2r}) \\
	                                        & -F_{r(n+1)}(((-1)^r-L_rQ_r+Q_{2r})P_r+L_rP_{2r}-P_{3r}) \\
	                                        & +(-1)^rP_{rn}(Q_rF_r-F_{2r})\\
	                                        & -P_{r(n+1)}(((-1)^r-Q_rL_r+L_{2r})F_r+Q_rF_{2r}-F_{3r}),
\end{split}
\end{equation}
\end{theorem}
where 
\[
\gamma(r) = 2 - ((-1)^r + (-1)^r)L_r Q_r + (-1)^r Q_{2r} + (-2)^r L_r^2.
\]
\begin{proof}
In Theorem~\ref{thm.rtf6ne1} set $X_m=F_m$, $Y_m=P_m$, ${V_X}_m=L_m$, ${V_Y}_m=Q_m$, $q_X=-1$ and $q_Y=-1$.
\end{proof}
\begin{example}
	For integer n, 
	\begin{align}
		\sum_{k=0}^n F_k P_{n-k} &= P_n - F_n, \\
		12\sum_{k=0}^n F_{2k} P_{2n-2k} &= P_{2n} - 2F_{2n}, \\
		106\sum_{k=0}^n F_{3k} P_{3n-3k} &= 10P_{3n} - 25F_{3n}.
	\end{align}
\end{example}

\begin{remark}
Seiffert \cite{Seiffert} has shown the similar convolution
\begin{equation*}
\sum_{k=0}^n F_{r(k+1)} P_{r(n+1-k)} = \frac{F_{r}P_{r(n+2)} - P_{r}F_{r(n+2)}}{2Q_r - L_r}.
\end{equation*}
\end{remark}

\subsection{Convolution of two Lucas sequences of the first kind}

\begin{theorem}\label{cor.4tzg84a}
For $j\in\mathbb Z$, let ${U_X}_j(p_X,q_X)$ and ${U_Y}_j(p_Y,q_Y)$ be two Lucas sequences of the first kind with respective 
associated Lucas sequences of the second kind ${V_X}_j(p_X,q_X)$ and ${V_Y}_j(p_Y,q_Y)$. Then,
\begin{equation}
\begin{split}
&\gamma(r) \sum_{k = 0}^n {U_X}_{rk} {U_Y}_{r(n - k)} \\
&\qquad = q_X^r {U_X}_{rn} \left( {V_X}_r  {U_Y}_r  - {U_Y}_{2r} \right) - {U_X}_{r(n + 1)} 
\left( (q_X^r - {V_X}_r {V_Y}_r + {V_Y}_{2r}){U_Y}_r + {V_X}_r {U_Y}_{2r} - {U_Y}_{3r} \right)\\
&\qquad\, + q_Y^r {U_Y}_{rn} \left( {V_Y}_r {U_X}_r - {U_X}_{2r} \right) - {U_Y}_{r(n + 1)} \left( (q_Y^r - {V_Y}_r {V_X}_r 
+ {V_X}_{2r}){U_X}_r  + {V_Y}_r {U_X}_{2r} - {U_X}_{3r} \right);
\end{split}
\end{equation}
where
\[
\gamma (r) = q_X^{2r} + q_Y^{2r} - (q_X^r + q_Y^r ){V_X}_r {V_Y}_r + q_X^r {V_Y}_{2r} + q_Y^r {V_X}_r^2. 
\]
\end{theorem}

\begin{theorem}
For integers $r$ and $n$,
\begin{equation}
\begin{split}
	&\gamma(r) \sum_{k = 0}^n P_{rk} J_{r(n - k)} \\
	&\qquad = (-1)^r P_{rn}(Q_r J_r-J_{2r}) - P_{r(n+1)}(((-1)^r - Q_r j_r + j_{2r})J_{2r} + Q_r J_{2r} - J_{3r}) \\
	&\qquad\, +(-2)^r J_{rn}(j_r P_r - P_{2r}) - J_{r(n+1)}(((-2)^r - j_r Q_r + Q_{2r})P_r + j_r P_{2r} - P_{3r});
\end{split}
\end{equation}
where 
\[
\gamma(r) = 1 + 4^r -((-1)^r+(-2)^r) Q_r j_r + (-1)^r j_{2r} + (-2)^r Q_r^2.
\]
\end{theorem}
\begin{example}
	\begin{align}
		2\sum_{k=0}^n P_k J_{n-k} &= 2J_n - J_{n+1} - P_n - P_{n+1}, \\
	   	28\sum_{k=0}^n P_{2k} J_{2n-2k} &= P_{2n} + 51P_{2n+2} - 6J_{2n+2} - 8J_{2n}, \\
	   	686\sum_{k=0}^n P_{3k} J_{3n-3k} &= 21P_{3n} - 591P_{3n+3} - 35J_{3n+3} - 280J_{3n}.		
	\end{align}
\end{example}

\subsection{Convolution of two Lucas sequences of the second kind}

\begin{theorem}
For $j\in \mathbb{Z}$, let $V_{X_j}(p_X,q_X)$ and $V_{Y_j}(p_Y,q_Y)$ be two Lucas sequences of the second kind. Then, 
\begin{equation}
\begin{split}
\gamma(r) \sum_{k = 0}^n V_{X_{rk}} V_{Y_{r(n - k)}} &= q_X^r V_{X_{rn}} \left( 2(q_X^r - {V_X}_r {V_Y}_r + {V_Y}_{2r} ) 
+ {V_X}_r V_{Y_{r}} - V_{Y_{2r}} \right) \\
		&\quad - V_{X_{r(n + 1)}} \left( {(q_X^r - {V_X}_r  {V_Y}_r + {V_Y}_{2r}  )V_{Y_{r}} + {V_X}_r V_{Y_{2r}} - V_{Y_{3r}} } \right)\\
		&\quad\, + q_Y^r V_{Y_{rn}} \left( {2(q_Y^r - {V_Y}_r {V_X}_r + {V_X}_{2r} ) + {V_Y}_r V_{X_{r}} - V_{X_{2r}} } \right)\\
		&\quad\; - V_{Y_{r(n + 1)}} \left( {(q_Y^r - {V_Y}_r {V_X}_r + {V_X}_{2r} )V_{X_{r}} + {V_Y}_r V_{X_{2r}} - V_{X_{3r}} } \right),
\end{split}
\end{equation}
where
\begin{equation*}
\gamma (r) = q_X^{2r} + q_Y^{2r} - (q_X^r + q_Y^r ){V_X}_r {V_Y}_r + q_X^r {V_Y}_{2r} + q_Y^r {V_X}_r^2. 
\end{equation*}
\end{theorem}

\begin{theorem}
For integers $r$ and $n$,
\begin{equation}
\begin{split}
\gamma(r) \sum_{k = 0}^n L_{rk} j_{r(n - k)} &= (-1)^r L_{rn}(2((-1)^r-L_rj_r+j_{2r})+L_r j_r-j_{2r}) \\
&\qquad - L_{r(n+1)}(((-1)^r - L_r j_r + j_{2r})j_{r} + L_r j_{2r} - j_{3r}) \\
&\qquad +(-2)^r j_{rn}(2((-2)^r-j_rL_r+L_{2r})j_r L_r - L_{2r}) \\
&\qquad - j_{r(n+1)}(((-2)^r - j_r L_r + L_{2r})L_r + j_r L_{2r} - L_{3r}),
\end{split}
\end{equation}
where 
\begin{equation*}
\gamma(r) = 1 + 4^r -((-1)^r+(-2)^r) L_r j_r + (-1)^r j_{2r} + (-2)^r L_r^2.
\end{equation*}
\end{theorem}
\begin{example}
	\begin{align}
		\sum_{k=0}^n L_k j_{n-k} &= 4j_n +j_{n+1} - 2L_n - L_{n+1}, \\
		5\sum_{k=0}^n L_{2k} j_{2n-2k} &=5j_{2n+2}+L_{2n+2}-4L_{2n}, \\
		140\sum_{k=0}^n L_{3k} j_{3n-3k} &= 35L_{3n} + L_{3n+3} +42j_{3n+3} - 176j_{3n}.		
	\end{align}
\end{example}

\section{More convolutions from a simple relation}

In this section we state a slight generalization of an identity of Carlitz \cite{Carlitz}.
The idea can also be applied to Chebyshev polynomials. 

\begin{theorem}\label{thm_carlitz}
Let $U_n, V_n, t_n(x)$, and $u_n(x)$ be the Lucas sequences and Chebyshev polynomials of the first and the second kinds, respectively.
Let further $T(n,k),0\leq k\leq n$ be a sequence such that $T(n,k)=T(n,n-k)$. Then we have the convolution identities
\begin{equation}
\sum_{k=0}^n T(n,k) U_{rk} V_{r(n-k)} = U_{rn} \sum_{k=0}^n T(n,k)
\end{equation}
and
\begin{equation}
\sum_{k=0}^n T(n,k) u_{rk-1}(x) t_{r(n-k)}(x) = \frac{u_{rn-1}(x)}{2} \sum_{k=0}^n T(n,k).
\end{equation}
\end{theorem}
\begin{proof}
Using the Binet forms we get
\begin{equation*}
\sum_{k=0}^n T(n,k) U_{rk} V_{r(n-k)} = U_{rn} \sum_{k=0}^n T(n,k) 
+ \frac{1}{\alpha - \beta} \sum_{k=0}^n T(n,k) \Big ( \alpha^{rk}\beta^{r(n-k)} - \alpha^{r(n-k)}\beta^{rk} \Big ).
\end{equation*}
Now, 
\begin{equation*}
\sum_{k=0}^n T(n,k) \alpha^{rk}\beta^{r(n-k)} = \sum_{k=0}^n T(n,n-k)\alpha^{r(n-k)}\beta^{rk} = \sum_{k=0}^n T(n,k)\alpha^{r(n-k)}\beta^{rk},
\end{equation*}
and therefore
\begin{equation*}
\sum_{k=0}^n T(n,k) \Big ( \alpha^{rk}\beta^{r(n-k)} - \alpha^{r(n-k)}\beta^{rk} \Big ) = 0.
\end{equation*}
The proof of the Chebyshev variant is very similar and omitted.
\end{proof}

\begin{example}
Let $T(n,k)=1$. Then we recover \eqref{UV_conv} and get for the Chebyshev polynomials
\begin{equation}
\sum_{k=0}^n u_{rk-1}(x) t_{r(n-k)}(x) = \frac{n+1}{2} u_{rn-1}(x).
\end{equation}
\end{example}

\begin{example}
Let $T(n,k)=k(n-k)$. Then we get
\begin{equation}
\sum_{k=0}^n k(n-k) U_{rk} V_{r(n-k)} = \frac{(n-1)n(n+1)}{6} U_{rn},
\end{equation}
and
\begin{equation}
\sum_{k=0}^n k(n-k) u_{rk-1}(x) t_{r(n-k)}(x) = \frac{(n-1)n(n+1)}{12} u_{rn-1}(x),
\end{equation}
where we used the simple evaluation
\begin{equation*}
\sum_{k=0}^n k(n-k) = \frac{(n-1)n(n+1)}{6}.
\end{equation*}
\end{example}

\begin{example}
Let $T(n,k)=k^2(n-k)^2$. Then we get
\begin{equation}
\sum_{k=0}^n k^2(n-k)^2 U_{rk} V_{r(n-k)} = \frac{n(n^4-1)}{30} U_{rn},
\end{equation}
and
\begin{equation}
\sum_{k=0}^n k^2(n-k)^2 u_{rk-1}(x) t_{r(n-k)}(x) = \frac{n(n^4-1)}{60} u_{rn-1}(x),
\end{equation}
where we used the simple evaluation
\begin{equation*}
\sum_{k=0}^n k^2(n-k)^2 = \frac{n(n^4-1)}{30}.
\end{equation*}
\end{example}

\begin{example}
Let $T(n,k)=U_{rk} U_{r(n-k)}, r\geq 1$. Then we get
\begin{equation}
\sum_{k=0}^n U_{rk}^2 U_{r(n-k)} V_{r(n-k)} = \frac{U_{rn}}{\Delta U_r}\Big ((n + 1)U_r V_{rn} - 2U_{r(n + 1)}\Big ) ,
\end{equation}
where \eqref{UU_conf} was used. Similarly, with $T(n,k)=V_{rk} V_{r(n-k)}, r\geq 1$, and using \eqref{VV_conf},
\begin{equation}
\sum_{k=0}^n U_{rk} V_{rk} V_{r(n-k)}^2 = \frac{U_{rn}}{U_r}\Big ((n + 1)U_r V_{rn} + 2U_{r(n + 1)}\Big ).
\end{equation}
\end{example}

\begin{example}
Let $T(n,k)=t_{rk}(x) t_{r(n-k)}(x)$. Then we get
\begin{equation}
\sum_{k=0}^n u_{rk-1}(x) t_{rk}(x) t_{r(n-k)}^2(x) = 
\frac{u_{rn-1}(x)}{4 u_{r-1}(x)}\Big ((n + 1)u_{r-1}(x) t_{rn}(x) + u_{r(n + 1)-1}(x)\Big ) ,
\end{equation}
where \eqref{eq.qa5lyvu} was used. Similarly, with $T(n,k)=u_{rk}(x) u_{r(n-k)}(x)$, and using \eqref{eq.h4s17al},
\begin{align}
& \sum_{k=0}^n u_{rk-1}(x) u_{rk}(x) u_{r(n-k)}(x) t_{r(n-k)}(x)  \nonumber \\
&\qquad\qquad = \frac{u_{rn-1}(x)}{4 (x^2-1)u_{r-1}(x)}\Big ((n + 1)u_{r-1}(x) t_{rn+2}(x) - u_{r(n + 1)-1}(x) \Big ).
\end{align}
\end{example}

We continue with some examples involving binomial coefficients, which have not been the subject of this paper
but give an idea of the broad applicability of Theorem \ref{thm_carlitz}. 

\begin{example}
Let $T(n,k)=\binom{n}{k}$. Then we get
\begin{equation}
\sum_{k=0}^n \binom{n}{k} U_{rk} V_{r(n-k)} = 2^n U_{rn},
\end{equation}
and
\begin{equation}
\sum_{k=0}^n \binom{n}{k} u_{rk-1}(x) t_{r(n-k)}(x) = 2^{n-1} u_{rn-1}(x).
\end{equation}
\end{example}

\begin{example}
Let $T(n,k)= \binom{n}{k}^2$. Then we get
\begin{equation}
\sum_{k=0}^n \binom{n}{k}^2 U_{rk} V_{r(n-k)} = \binom{2n}{n} U_{rn},
\end{equation}
and
\begin{equation}
\sum_{k=0}^n \binom{n}{k}^2 u_{rk-1}(x) t_{r(n-k)}(x) = \binom{2n}{n} \frac{u_{rn-1}(x)}{2},
\end{equation}
where we used the evaluation
\begin{equation*}
\sum_{k=0}^n \binom{n}{k}^2 = \binom{2n}{n}.
\end{equation*}
\end{example}

\begin{example}
Let $T(n,k)= \binom{2n}{2k}$. Then we get
\begin{equation}
\sum_{k=0}^n \binom{2n}{2k} U_{rk} V_{r(n-k)} = 2^{2n-1} U_{rn},
\end{equation}
and
\begin{equation}
\sum_{k=0}^n \binom{2n}{2k} u_{rk-1}(x) t_{r(n-k)}(x) = 4^{n-1} u_{rn-1}(x),
\end{equation}
where we used the evaluation
\begin{equation*}
\sum_{k=0}^n \binom{2n}{2k} = 2^{2n-1}.
\end{equation*}
\end{example}

\begin{example}
Let $T(n,k)= \binom{3n}{3k}$. Then we get
\begin{equation}
\sum_{k=0}^n \binom{3n}{3k} U_{rk} V_{r(n-k)} = \frac{2}{3} U_{rn} \Big ( 2^{3n-1}+(-1)^n\Big ),
\end{equation}
and
\begin{equation}
\sum_{k=0}^n \binom{3n}{3k} u_{rk-1}(x) t_{r(n-k)}(x) = \frac{1}{3} u_{rn-1}(x) \Big ( 2^{3n-1}+(-1)^n\Big ),
\end{equation}
where we used the evaluation
\begin{equation*}
\sum_{k=0}^n \binom{3n}{3k} = \frac{2}{3} \Big ( 2^{3n-1}+(-1)^n\Big ).
\end{equation*}
\end{example}

\begin{example}
Let $T(n,k)= \binom{n}{k}B_k(x)B_{n-k}(x)$, where $B_n(x)$ are the Bernoulli polynomials. Then we get
\begin{equation}
\sum_{k=0}^n \binom{n}{k} B_k(x)B_{n-k}(x) U_{rk} V_{r(n-k)} = U_{rn} \Big ( n(2x-1)B_{n-1}(2x) - (n-1)B_n(2x)\Big ),
\end{equation}
and
\begin{equation}
\sum_{k=0}^n \binom{n}{k} B_k(x)B_{n-k}(x) u_{rk-1}(x) t_{r(n-k)}(x) = \frac{u_{rn-1}(x)}{2}\Big ( n(2x-1)B_{n-1}(2x) - (n-1)B_n(2x)\Big ),
\end{equation}
where we used the polynomial identity \cite{Agoh}
\begin{equation*}
\sum_{k=0}^n \binom{n}{k} B_k(x)B_{n-k}(x) = n(2x-1)B_{n-1}(2x) - (n-1)B_n(2x).
\end{equation*}
See also \cite{frontczak2020} for a special case of the first convolution.
\end{example}

Many more examples could be stated here involving $\binom{n}{k} U_{mk} U_{m(n-k)}, \binom{n}{k} V_{mk}V_{m(n-k)}$, 
Catalan numbers, and other sequences. 

\section{Concluding remarks}

This paper offers a direct approach to convolutions of second order sequences,
without any use of generating functions. It allows in a straightforward manner to generalize
some fixed convolutions involving these sequences. Though not stated explicitly in the main text,
the approach is also applicable to derive convolutions for Fibonacci (Lucas) polynomials,
Pell (Pell-Lucas) polynomials and other polynomial families.

\end{document}